\documentclass[11pt]{article}

\usepackage[a4paper,margin=1in]{geometry} % margens razoáveis para arXiv
\usepackage{amsmath,amssymb,amsthm}
\usepackage{graphicx}
\usepackage{xspace}
\usepackage{xcolor}
\usepackage{hyperref}

\newtheorem{theorem}{Theorem}[section]
\newtheorem{lemma}[theorem]{Lemma}
\newtheorem{proposition}[theorem]{Proposition}

\DeclareMathOperator{\grad}{grad}  
\DeclareMathOperator{\dom}{dom}
\theoremstyle{remark}

\usepackage{amsthm}

\theoremstyle{definition}
\newtheorem{definition}{Definition}[section]

\makeatletter
\AtBeginDocument{%
  \renewcommand\subsection{\@startsection{subsection}{2}{\z@}%
    {-2.5ex \@plus -1ex \@minus -.2ex}%
    {1.3ex \@plus .2ex}%
    {\normalfont\bfseries}}%
  \renewcommand\subsubsection{\@startsection{subsubsection}{3}{\z@}%
    {-2.5ex \@plus -1ex \@minus -.2ex}%
    {1.3ex \@plus .2ex}%
    {\normalfont\bfseries}}%
}
\makeatother

\newcommand{\keywords}[1]{%
  \bigskip\noindent\textbf{Keywords: }#1}
\newcommand{\subclass}[1]{%
  \bigskip\noindent\textbf{MSC 2020: }#1}

\begin{document}

\title{An Adaptive Proximal Point Method for Nonsmooth and Nonconvex
Optimization on Hadamard Manifolds}

\author{%
  Vitaliano S.~Amaral\thanks{Department of Mathematics, Federal University of Piau\'i, Teresina, Piau\'i, Brazil. E-mail: \texttt{vitalianoamaral@ufpi.edu.br}} \and
  Marcio Ant\^onio de A.~Bortoloti\thanks{DCET, Universidade Estadual do Sudoeste da Bahia, Vit\'oria da Conquista, Bahia, Brazil. E-mail: \texttt{mbortoloti@uesb.edu.br}} \and
  Jurandir O.~Lopes\thanks{Department of Mathematics, Federal University of Piau\'i, Teresina, Piau\'i, Brazil. E-mail: \texttt{jurandir@ufpi.edu.br}} \and
  Gilson N.~Silva\thanks{Department of Mathematics, Federal University of Piau\'i, Teresina, Piau\'i, Brazil. E-mail: \texttt{gilson.silva@ufpi.edu.br}}%
}

\date{} % deixe vazio para arXiv

\maketitle

\begin{abstract}
This paper addresses a class of nonsmooth and nonconvex optimization problems defined on complete Riemannian manifolds. The objective function has a composite structure, combining convex, differentiable, and lower semicontinuous terms, thereby generalizing the classical framework of difference-of-convex programming. Motivated by recent advances in proximal point methods in Euclidean and Riemannian settings, we propose two variants: one that uses the Lipschitz constant of the gradient of the smooth part, suitable when this parameter is accessible, and another that dispenses with such knowledge, expanding its applicability. We analyze the complexity of both approaches, establish their convergence, and illustrate their effectiveness through numerical experiments.
\end{abstract}


\subclass{90C30; 90C26; 49Q99}

\section{Introduction}
In this paper, we study a class of nonsmooth and nonconvex optimization problems defined on a complete Riemannian, called in this introduction of {\it NNOP-CR}.
The objective function has a composite structure, where one term is convex, another term is continuously differentiable with a Lipschitz continuous gradient, and a third term is proper and lower semicontinuous but not necessarily convex.
% consider the Riemannian nonsmooth and nonconvex optimization problem written as:
% \begin{equation}\label{dcg}
% 	\min_{x \in {\cal M}}\{f(x):=g_1(x)+g_2(x)-h(x)\},
% \end{equation}
% where ${\cal M}$ is a complete Riemannian manifold, $h:{\cal M}\to \mathbb{R}$ is a convex function, $g_2: {\cal M}\to \mathbb{R}$ is a continuously differentiable function with its gradient being Lipschitz continuous over ${\cal M}$,  and $g_1:{\cal M} \to \mathbb{R}\cup\{\infty\}$ is a  proper and lower semicontinuous function. 
When the sum of the nonsmooth and the differentiable terms is assumed to be convex, the problem reduces to the well-known framework of difference-of-convex (DC) programming, which has been widely investigated in the literature. However, by allowing the nonsmooth component to be nonconvex, we obtain a broader and more challenging class of problems that generalizes the classical DC setting.
Such formulations naturally arise in many areas of applied mathematics and engineering. 
Examples include machine learning \cite{aragon2024coderivative}, image processing image processing \cite{aragon2024coderivative}, compressed sensing \cite{yin2015minimization}, location optimization \cite{mordukhovich2014easy}, biochemistry modeling \cite{aragon2024coderivative}, sparse optimization \cite{gotoh2018dc}, clustering methods \cite{aragón2020boosted,bajaj2022solving}, hierarchical clustering \cite{nam2018nesterov}, clusterwise linear regression \cite{bagirov2018nonsmooth}, multicast network design \cite{geremew2018dc}, and multidimensional scaling \cite{an2001dc}.

% , compressed sensing, location analysis, biochemical modeling, sparse optimization, clustering methods, hierarchical clustering, clusterwise regression, multicast network design, and multidimensional scaling.

% Important examples in optimization are modeled as in \eqref{dcg}, for example in, machine learning \cite{aragon2024coderivative}, 

% ?????????
% If in  \eqref{dcg} we assume that $g:=g_1+g_2:{\cal M}\to \mathbb{R}\cup\{\infty\}$ is a convex function, then  it becomes the classical
% % $$
% % \min_{x \in {\cal M}}\{f(x)=g(x)-h(x)\} 
% % $$
% difference of convex problem (DC, for short), see for example the references \cite{Almeida2020,SouzaOliveira}. 
% Hence, in view of the nonconvexity of function $g_1$, problem \eqref{dcg} is more general than the classical DC problem.
% The problema \eqref{dcg} when is $h\equiv 0$ and $g_2$ is continuously differentiable was studied by Fukushima and Mine \cite{FUKUSHIMA1981} in Euclidean spaces.
% In recent decades, competitive algorithms for solving DC problems has been developed. 
% This kind of problem has gained significant recognition in optimization because its applications covers 

% One popular method for solving \eqref{dcg} is the proximal point method (PPM for short), which first appeared in Martinet's work \cite{Martinet1970}. Some years later, PPM also appeared in the Rockafellar's work \cite{Rockafellar1976} in the study of monotone operators. The global and local convergence rate of PPM have been extensively studied; see some studies \cite{Almeida2020,Andrade2023,Hare2009,Pan2007} and the references therein. 
We now focus our discussion on the efficient methods applied for solving problems belonging to {\it NNOP-CR} in the Euclidian context. 
A popular and probably the first algorithm developed to solve DC problems was the Difference of Convex Algorithm (DCA), see \cite{tao1997convex,tao1986algorithms}. 
A significant progress in this field was the development of the Boosted DC Algorithm (BDCA) \cite{aragon2018accelerating,aragón2020boosted,artacho2019boosted}, which enhances the convergence speed of the traditional DCA through the integration of a line search strategy; see also \cite{aragon2024coderivative,zhang2024boosted} for other recent contributions. 
This strategy enables larger step sizes compared to the standard DCA, leading to improved convergence efficiency and better solution quality by avoiding suboptimal local minima. 
The initial BDCA was mainly designed for scenarios where the first DC component is differentiable, ensuring a descent direction.
However, its original formulation has limitations when both components are nondifferentiable. 

To overcome the challenge that BDCA encounters as mentioned in the previous paragraph, the recent work \cite{ferreira2024boosted} introduced the nonmonotone Boosted Difference of Convex Algorithm (nmBDCA). 
The key difference between nmBDCA and BDCA arises when the first component is nondifferentiable.
In such situations, the line search direction utilized in BDCA may not guarantee a descent direction, rendering the standard Armijo line search ineffective.
To address this issue, nmBDCA incorporates a nonmonotone Armijo-like line search that remains effective even when the first component is not differentiable. 
This adjustment enables nmBDCA to tackle a wider range of problems, thereby broadening its applicability beyond BDCA and improving the efficiency of the line search. 
% Additionally, enhancing BDCA’s overall performance remains a possibility by tackling the computational difficulties associated with solving the subproblem at each iteration. Due to finite precision limitations, obtaining exact solutions is often impractical, leading to the necessity of solving these subproblems approximately in practice.
Further developments beyond BDCA and nmBDCA were recently explored in \cite{ferreira2024inexact}, where the subproblem at each iteration of its inexact nmBDCA is solved approximately using a relative error tolerance.

We now discuss on the recent developments on Riemannian DC problems. The crescent interest in optimization methods on Riemannian manifolds is due to the fact that practical applications arise whenever the
natural structure of the data is modeled as an optimization problem on a Riemannian
manifold. For instance see \cite{bacak2016second,bergmann2024difference,bergmann2016parallel,bhattacharya2008statistics,esposito2019total,gotoh2018dc,park1995lie,weinmann2014total,zhou2023semismooth}. 
As far as we know, the study presented in \cite{SouzaOliveira} was the first to address DC functions within Riemannian manifolds. 
Specifically, the authors in \cite{SouzaOliveira} introduced the proximal point algorithm for DC functions and analyzed its convergence in the context of Hadamard manifolds. 
More recently, \cite{Almeida2020} proposed a refined version of the algorithm in \cite{SouzaOliveira} in the same Riemannian framework, aiming to enhance the convergence speed of the method introduced in \cite{SouzaOliveira}. 
In the recent paper \cite{gotoh2018dc}, a DC algorithm is presented for sparse optimization problems having a cardinality or rank constraint. 
The main idea in \cite{gotoh2018dc} is to transform the cardinality-constrained problem into a penalty function form and derive exact penalty parameter values for some optimization problems,
especially for quadratic minimization problems which often appear in practice. Thus, the new problem belongs to {\it NNOP-CR}, i.e., a DC problem with three nonlinear functions.
% and the references therein. To the best of our knowledge, the PPM for DC optmization was first introduced by Sun et al. \cite{Sun2003}. Their method consist to increasing the component function $h$ along the direction of the subgradient and then decreasing the function $f$ by using a proximal step applied to the component function $g$. Another proof for convergence analysis of this method was proposed by Moudafi and Maing\'e~\cite{Moudafi2006}.
More relevant researches about the proximal point methods within the framework of Riemannian manifolds 
% have garnered significant attention in recent years, reflecting their growing importance in mathematical optimization and related fields; see, for example, 
can be found in \cite{Andrade2023,BENTO2010564,Bento2017,DaCruzNeto2006,Ferreira01042002,Bento01022015,Li2009,Wang2016} and in \cite{Almeida2020,Baygorrea2016,cruz2020generalized,PapaQuiroz2012,PapaQuiroz2009,SouzaOliveira}. 
% The proximal point method, applied to convex minimization problems in the context of Riemannian manifolds, was initially studied by Ferreira and Oliveira \cite{Ferreira01042002}. 
% Following the publication of these papers, several other works have been developed to address minimization problems for various classes of objective functions, among these, two notable classes: quasiconvex functions (see, for example 
% \cite{Baygorrea2016,PapaQuiroz2012,PapaQuiroz2009}) and DC (Difference of Convex) functions (see, for example \cite{Almeida2020,cruz2020generalized,SouzaOliveira})

In the case where the Riemannian manifold is the Euclidian space, in \cite{An02012017} is proposed a PPM to solve a problem in {\it NNOP-CR}, under the same conditions as mentioned before.
The method proposed in \cite{An02012017} is based on a proximal-point framework. At each iteration, it solves a strongly convex subproblem that involves both the nonsmooth and smooth components of the objective. The construction of these subproblems requires knowledge of the Lipschitz constant associated with the gradient of the differentiable term.
% The method proposed in \cite{An02012017} uses the iteration for $k=1, 2, \ldots,$
% 	\begin{equation} \label{sub1}
% 	x^{k+1}\in \arg\min_{x\in\mathbb{R}^n}\bigg\{g_1(x)-\langle y^k-\nabla g_2(x^k),x-x^k\rangle +\frac{t}{2}\|x-x^k\|^2\bigg\},
% 	\end{equation}
% with $y^k\in \partial h(x^k)$ and $t>L$, where $L>0$ is the Lipschitz constant for the gradient of $g_2$. 
Similarly, in \cite{Amaral2024} the authors proposed an algorithm based on a proximal gradient method for a particular case  by assuming that the smooth function has a gradient Hölder. 
% Where they showed that the number of iterations needed to reach an approximate solution is $\mathcal{O}(\max \{\epsilon^{ -\frac{ \beta+1}{\beta}} , \eta^\frac{\beta}{2}\epsilon^{ - \frac{2(\beta+1)}{\beta} }\})$, where $\beta$ is the Holder exponent and $\eta$ is an estimate for the subgradient of $h$.

In this paper, we present two variants for the proximal point methods for solving {\it NNOP-CR}. In our convergence analysis, we require that only one of the functions is convex. 
In the first method, prior knowledge of the Lipschitz constant associated to the gradient of the smooth part is required, making it more suitable for problems in which this constant is easy to compute. 
However, determining such a constant is not always straightforward. 
Therefore, in this paper, we propose a second method that does not require prior knowledge of the Lipschitz constant. % and its subproblem can be solved inexactly.
We analyze the complexity of both methods and prove that they reach an approximate solution in a finite number of iterations, with complexity of order $\mathcal{O}(\epsilon^{-2})$, with $\epsilon>0$ to be specified later.

The remainder of this paper is organized as follows. In Section~\ref{preliminar}, we present preliminary results to support the development of the proposed methods. 
Section~\ref{method1} introduces the first version of the Riemannian proximal point method and demonstrates its effectiveness.
In Section~\ref{adaptive}, we propose a second version of the method, which does not require prior knowledge of the Lipschitz constant, and analyze its worst-case iteration complexity. 
Section~\ref{kl} is devoted to the convergence analysis under the assumption that the objective function satisfies the Kurdyka-Łojasiewicz (KL) property. 
In Section~\ref{secexp}, we  illustrate the applicability of the proposed methods by provideing numerical experiments.
Finally, Section~\ref{Conclusions} presents the conclusions of the paper.

\section{Preliminaries}\label{preliminar}
  Throughout this work, ${\cal M}$ will denote a finite-dimensional Riemannian manifold, 
and $T_x{\cal M}$ will represent the tangent space of ${\cal M}$ at the point $x$.  
The norm induced by the Riemannian metric $\langle \cdot , \cdot \rangle$ will be denoted by $\|\cdot\|$.

By definition, a Hadamard manifold is a complete, simply connected Riemannian manifold with nonpositive sectional curvature.  
In this work, we restrict our analysis to manifolds of this type.

The Riemannian metric provides the following fundamental notions:

\begin{itemize}
  \item[(i)] Let $\gamma:[a,b]\to{\cal M}$ be a piecewise smooth curve with $\gamma(a)=x$ and $\gamma(b)=y$.  
  Its length is given by
  $$
     L(\gamma) = \int_a^b \|\gamma'(t)\|\, dt.
  $$

  \item[(ii)] The Riemannian distance between $x$ and $y$ is defined by
  $$
     d(x,y) := \inf_{\gamma} L(\gamma),
  $$
  where the infimum is taken over all such curves.  
  This distance induces the topology of ${\cal M}$.

  \item[(iii)] Given a geodesic $\gamma$ joining $x$ to $y$, the parallel transport along $\gamma$ is denoted by  
  $\Gamma_{y,x}: T_x{\cal M}\to T_y{\cal M}$.

  \item[(iv)] A vector field $V$ along $\gamma$ is parallel if $\nabla_{\gamma'}V=0$.  
  In particular, $\gamma$ is a geodesic when its velocity field $\gamma'$ is parallel.

  \item[(v)] For any $a,b\in [a,b]$, the parallel transport  
  $\Gamma_{\gamma,\gamma(b),\gamma(a)}: T_{\gamma(a)}{\cal M}\to T_{\gamma(b)}{\cal M}$  
  is an isometry.
\end{itemize}

The completeness of the Hadamard manifold ${\cal M}$ allows the definition of the exponential map  
$$
\exp_x: T_x{\cal M} \to {\cal M}, \qquad \exp_x(v) = \gamma_v(1,x),
$$
for all $x \in {\cal M}$.  
Consequently, the inverse $\exp_x^{-1}: {\cal M} \to T_x{\cal M}$ is smooth ($C^\infty$). Since $d(x,y) = \|\exp_{y}^{-1}(x)\|$, the function $\rho_{y}(x) = \frac{1}{2} d^2(x,y)$ is $C^\infty$, and its gradient at $x$, $\grad \rho_{y}(x)$, is given by $\grad \rho_y(x) = -\exp_x^{-1}y,
$ as shown in \cite[Proposition 3.3]{Ferreira01042002}.

A classical result on Hadamard manifolds is the ``Comparison Theorem'' for geodesic triangles  
(see \cite[Proposition~4.5]{Sakai}), which states that, for any $x,y,z \in {\cal M}$,
\begin{equation}\label{eq:tri}
   d^2(x,y) + d^2(x,z) - 2 \langle \exp_x^{-1}(y), \exp_x^{-1}(z) \rangle \;\leq\; d^2(y,z).
\end{equation}

\medskip

Next, we recall some notions of convex analysis (see \cite{bento2018proximal,ledyaev2007nonsmooth,azagra2005nonsmooth}).  

Let $\xi:{\cal M}\to \mathbb{R}\cup\{+\infty\}$.  
Its domain is defined as
$$
   \dom\,\xi := \{x\in {\cal M} : \xi(x)<+\infty\}.
$$
We say that $\xi$ is \emph{proper} if $\dom\,\xi \neq \emptyset$,  
and that it is \emph{lower semicontinuous} (lsc) at $x\in{\cal M}$ if
$$
   \xi(x)\leq \liminf_{y\to x}\xi(y).
$$  
The Fr\'echet subdifferential of $\xi$ at $x\in\dom\,\xi$ is defined as
$$
   \partial^F \xi(x) :=
   \Big\{d\psi_x : \psi \in C^1({\cal M}),\; \xi-\psi \text{ attains a local minimum at } x \Big\},
$$
where $d\psi_x \in (T_x{\cal M})^*$ is given by
$$
   d\psi_x(v) = \langle \grad \psi(x), v \rangle, \quad v\in T_x{\cal M}.
$$
The limiting subdifferential of $\xi$ at $x \in {\cal M}$ is defined as
$$
   \partial^L \xi(x) :=
   \Big\{v\in T_x{\cal M} : \exists (x^k,v^k)\in Gr(\partial^F \xi),\;
   (x^k,v^k)\to (x,v),\; \xi(x^k)\to \xi(x) \Big\},
$$
where
$$
   Gr(\partial^F \xi) := \{(y,u)\in T{\cal M} : u\in \partial^F \xi(y)\}.
$$

\medskip

It holds that $\partial^F \xi(x) \subset \partial^L \xi(x)$.  
Moreover, if $\xi$ is differentiable at $x\in{\cal M}$, then
$$
   \partial^F \xi(x) = \partial^L \xi(x) = \{\grad \xi(x)\}.
$$
Finally, if $\xi$ attains a local minimum at $x$, then $0 \in \partial^L \xi(x)$.  

\noindent
In the convex case, we also recall some classical results of convex analysis.
Given a convex function $\xi:{\cal M}\to\mathbb{R}$, the \emph{subdifferential} of $\xi$ at $x\in{\cal M}$ is defined by
$$
   \partial \xi(x) := \Big\{ u \in T_x{\cal M} \;:\; 
   \langle u, \exp_x^{-1}(y)\rangle \leq \xi(y) - \xi(x), \;\; \forall y\in{\cal M}\Big\}.
$$
In this case, for every $x\in{\cal M}$ it holds that
$$
   \partial^F \xi(x) = \partial^L \xi(x) = \partial \xi(x).
$$

These definitions lead to several fundamental properties, which we summarize below.

\begin{itemize}
    \item[(a)] for every $x\in{\cal M}$ the set $\partial \xi(x)$ is nonempty, convex, and compact. See \cite[Theorems~4.5--4.6]{Udriste}.
 \item[(b)]  consider a sequence $x^k\to x^*$. 
If $v^k\in \partial \xi(x^k)$ for all $k\in\mathbb{N}$, then the sequence $\{v^k\}$ is bounded and every accumulation point of $\{v^k\}$ belongs to $\partial \xi(x^*)$.
\item[(c)] Let $\varphi = \xi+g$, where $\varphi:{\cal M}\to\mathbb{R}\cup\{+\infty\}$ is proper and lsc,  
and $g:{\cal M}\to\mathbb{R}$ is continuously differentiable in a neighborhood of $x^*\in \dom\varphi$.  
Then
\begin{equation*}\label{soma}
   \partial^L \varphi(x^*) \subseteq \partial^L \xi(x^*) + \grad g(x^*).
\end{equation*}
\end{itemize}For more details, see \cite[Theorem~4.13]{ledyaev2007nonsmooth}.
The next result is fundamental for the definition of critical points. Its proof can be found in  \cite[Theorem~5.2]{karkhaneei2019nonconvex}. 

\begin{theorem}\label{teo:crit}
If $x^*\in \dom f$ is a local minimizer of Problem \eqref{dcg}, then
\begin{equation}\label{eq:crit}
   \partial h(x^*) \subseteq \partial^L g_1(x^*) + \grad g_2(x^*).
\end{equation}
\end{theorem}
% \begin{proof}
% See \cite[Theorem~5.2]{karkhaneei2019nonconvex}. \qed
% \end{proof}

A point $x^*\in \dom f$ that satisfies \eqref{eq:crit} is called a \emph{stationary point} of Problem \eqref{dcg}.  
Since this condition may be difficult to verify, one usually adopts the relaxed condition
\begin{equation}\label{critico}
 \big(\partial^L g_1(x^*)+\grad g_2(x^*)\big)\cap \partial h(x^*) \neq \emptyset,
\end{equation}
in which case $x^*$ is referred to as a \emph{critical point} of $f$.  
We denote by $S$ the set of all critical points of $f$ and assume, throughout this work, that $S\neq\emptyset$.  

\begin{proposition}\label{prop5}
Let $\zeta:{\cal M}\to\mathbb{R}$ be a differentiable function whose gradient mapping is Lipschitz continuous with constant $L>0$.  
Then, for every $p\in{\cal M}$ and $v\in T_p{\cal M}$,
\begin{equation}\label{eq:lipshitz}
   \zeta(\exp_p(v)) \leq \zeta(p) + \langle \grad \zeta(p),v\rangle + \tfrac{L}{2}\|v\|^2.    
\end{equation}
\end{proposition}

\begin{proof}
See \cite[Lemma~2.1]{Bento2017}. 
\end{proof}

\section{The First Riemannian Proximal  Point Method}\label{method1}
In this section, we present the first Riemannian proximal  point method (R-PPM) for solving the following problem:
\begin{equation}\label{dcg}
	\min_{x \in {\cal M}}\{f(x):=g_1(x)+g_2(x)-h(x)\},
\end{equation}
where ${\cal M}$ is a complete Riemannian manifold, $h:{\cal M}\to \mathbb{R}$ is a convex function, $g_2: {\cal M}\to \mathbb{R}$ is a continuously differentiable function with its gradient being Lipschitz continuous over ${\cal M}$,  and $g_1:{\cal M} \to \mathbb{R}\cup\{+\infty\}$ is a  proper and lower semicontinuous function.

Throughout it, we assume the following assumptions.
% in the remainder of this article, which is a mild hypothesis for our method that we will define next is well defined.
\begin{itemize}
    \item[(A1)]
   For each $y \in {\cal M}$ fixed and some $\gamma\in \mathbb{R}$, we have that 
\[\displaystyle\liminf_{d(x,y)\rightarrow +\infty}\frac{g_1(x)}{d(x,y)}>\gamma.
\]
\item[(A2)] Function $f$ is bounded from below,  i.e., there exists $f^{\text{low}}\in \mathbb{R} $ such that $f(x)\geq f^{\text{low}},$ for every $x \in \mathcal{M}$.
\end{itemize}\
   % \end{assumption}
We note that if $g_1$ is bounded from below or convex, Assumption (A1) is satisfied.
Next, we formally describe R-PPM.
\vspace{0.5cm}
\noindent
	\hrule
	\vspace{0.2cm}
	\noindent
	{\bf Algorithm R-PPM}
	\vspace{0.1cm}
	\hrule
\begin{itemize}
    \item [ ] {\bf Step~0.} Given $(x^0, \alpha, L) \in \dom\,f\times \mathbb{R}_{++}\times \mathbb{R}_{++}$, a sequence of positive
numbers $\{\lambda_k\}$ such that $0< L <\lambda_k\leq \lambda_{k+1} \leq L+\alpha$ and set $k:=0.$   
\item [ ] {\bf Step~1.} %Given $x^k\in \dom\, f$, 
Calculate 
\begin{equation}\label{wk}
v^k:=\grad g_2(x^k),~w^k \in \partial h(x^k), ~z^k:=\exp_{x^k}{\bigg[\frac{1}{\lambda_k}(w^k-v^k)\bigg]}. 
\end{equation}
\item [ ] {\bf Step~2.} Compute
\begin{equation}\label{xk}
x^{k+1} \in \arg\min_{x\in \mathcal{M}} \left\{f_k(x):=g_1(x)+\frac{\lambda_k}{2}d^2(x,z^k)\right\}. 
\end{equation}
\item [ ] {\bf Step~3.} If $x^{k+1} = x^k$, stop. Otherwise, set $k:=k+1$ and return to {\bf Step~1}.
\end{itemize}
\hrule

\noindent
\vspace{0.5cm}

% Before presenting the main result of this section,
We now make some remarks about R-PPM.
 First, we assume that the Lipschitz constant $L$ satisfying \eqref{eq:lipshitz} can be computed, that is, we suppose that $L$ is known. 
 Second, the real sequence $\{\lambda_k\}$ in Step~0 can be any nondecreasing sequence belonging to the interval $(L, L+\alpha]$. 
 Third, after finding $v^k$, $w^k$ and $z^k$ as in  \eqref{wk}, for a given starting point $x^0\in \dom f,$ the next point is generated by solving the Subproblem \eqref{xk}, which we assume to be easy to solve.

 Note that when $g_2\equiv 0$,  R-PPM becomes exactly the algorithm DCPPA proposed in \cite{SouzaOliveira}, even so, our problem is more general, as we consider that only the function \( h \) is convex.

 %{\color{red}Fourth, the stop criterion we consider in the convergence analysis of R-PPM is $x^{k+1}=x^k.$ But it would be interesting to test R-PPM with the stopping criterio relationed with the functional values of $f,$ that is, to test if $f(x^{k+1})-f(x^k)\le \epsilon,$ for some $\epsilon>0$ given.} 

%We end this  section by presenting our main result about the sequence $\{x^k\}$ generated by R-PPM. Its proof is  presented in
%Section~\ref{TR}.

%\begin{theorem}\label{mai:result1} Let $\{x^k\}$ be a sequence generated by  method R-PPM. Then, the following statements hold:
%\begin{enumerate} 
%\item[(a)] The sequence $\{f({{x}}^k)\}$ is monotonically decreasing. As a consequence, it is convergent;
%\item[(b)]$\displaystyle\sum_{k=0}^{\infty} d^2(x^{k+1},x^k)<\infty$. In particular, $\{x^{k+1}-x^k\} \to 0,$ as $k\to+\infty$;
%\item[(c)] Every cluster point of $\{x^k\}$ is a critical point of $f$.
%\item[(d)] \textcolor{red}{convegence rate as in my class ???}
%\end{enumerate}
%\end{theorem}

%\textcolor{red}{Escrever alguns comentários sobre o teorema acima}
%\section{Thecnical Results}\label{TR}
We establish some preliminaries results about the sequence $\{x^k\}$ generated by  R-PPM. First, we prove that the sequence $\{x^k\}$ is well-defined, along with proposing a practical stopping criterion for R-PPM.

\begin{proposition}\label{prop:well:defined} Let $\{x^k\}$ be the sequence generated by R-PPM. Then, the following statements hold:
\begin{enumerate} 
\item[(a)] $\{x^k \}$ is well-defined;
\item[(b)] If $x^{k+1} = x^k$, then R-PPM stops at a solution of (\ref{dcg}).
\end{enumerate}
\end{proposition}
\begin{proof}
$(a)$ Let $x^0 \in \dom\,f$ be the input for R-PPM. For some $k\ge 1,$ we assume that R-PPM reached the iteration $k$, and we will show  that the $(k + 1)$-th iterate exists. 
Since $g_2$ is differentiable and $h$ is convex, the well definition of the sequence $\{z^k\}$ in \eqref{wk} follows directly from the
fact that ${\cal M}$ is a complete Hadamard manifold. 
Now, note that the assumptions that $g_1$ is proper and lsc imply that $f_k(\cdot),$ defined in \eqref{xk}, is proper and  lsc as well. Furthermore, Assumption (A1) implies that if $d(x,z^k)$ is large enough, then we obtain that  $g_1(x)> \gamma d(x,z^k)$, for some $\gamma\in \mathbb{R}$ and $x\in {\cal M}.$ As a consequence, we use \eqref{xk} to conclude that  
\begin{eqnarray*}
f_k(x)&\geq& \gamma d(x,z^k)+\frac{\lambda_k}{2}d^2(x,z^k). 
\end{eqnarray*}
Thus, since $\lambda_k>0,$ we can conclude that $f_k(\cdot)$ is coercive. As a result, we can conclude that the function $f_k(\cdot)$ has  minimizers and we will denote one of them by $x^{k+1}$, where $x^{k+1} \in \dom\,f=\dom\,f_k$. Hence, from any $x^k\in \dom\,f$ the R-PPM generates $x^{k+1} \in \dom\,f.$  

$(b)$ Firstly, we define $p^k:=\lambda_k \exp_{x^k}^{-1} z^k.$ Then, it follows from (\ref{wk}), (\ref{xk}) and $\grad \rho_{x'}(x) = -\exp_x^{-1}(x'),
$ that
\begin{equation}\label{eq:cons:opt}
    w^k-v^k= p^k~\mbox{and}~0\in \partial^L g_1(x^{k+1})-\lambda_k \exp_{x^{k+1}}^{-1} z^k.
\end{equation} 
Let $s^{k+1} \in \partial^L g_1(x^{k+1}).$ If  $x^{k+1}=x^k$, then the previous inclusion and \eqref{eq:cons:opt} imply that $s^{k+1}-p^k=0.$ As a consequence, by using the identity in \eqref{eq:cons:opt} we conclude that $s^{k+1}+v^k=w^k.$ Therefore, $x^k$ is a critical point of $f$. 
\end{proof}
From now on, we assume that $x^{k+1}\neq x^k$ for all $k\ge 0;$ otherwise, Proposition~\ref{prop:well:defined}(b) ensures that R-PPM stops at a solution of Problem (\ref{dcg}).

\begin{proposition}\label{pro1} The sequence $\{x^k\}$ generated by R-PPM satisfies 
 \begin{equation*}\label{lem:mon:decre}
 f({x}^{k+1})\leq f(x^k)-\bigg(\frac{\lambda_k-L}{2}\bigg)d^2(x^{k+1},x^k), \quad\text{for each}\quad k\ge 1.
 \end{equation*}
\end{proposition}
\begin{proof} We first observe that (\ref{xk}) implies that $f_k(x^{k+1})\leq f_k(x),$ for each $x \in {\cal M}$ and $k\ge 1.$ In particular, for $x=x^k$, we have
\begin{equation}\label{eqq1}
g_1({x}^{k+1})+\frac{\lambda_k}{2}d^2(x^{k+1},z^k)\leq g_1({x}^{k})+ \frac{\lambda_k}{2}d^2(x^{k},z^k).
	\end{equation}
Due to the convexity of the function $h$, the following inequality holds for each $w^k\in \partial h(x^k)$:
\begin{equation}\label{eqq2}
\langle w^k,\exp^{-1}_{x^k}x^{k+1}\rangle - h({x}^{k+1}) \leq  - h(x^k).
\end{equation}
On the other hand, Proposition \ref{prop5} with $\xi=g_2,\; v=\exp^{-1}_{x^k}x^{k+1}$ and $p=x^k$ yields %for the function $g_ 2$ to  $v=\exp^{-1}_{x^k}x^{k+1}$ and $p=x^k$, we get
\begin{equation}\label{eqq2a}
 g_2(x^{k+1}) \leq g_2(x^k)+ \langle v^k,\exp^{-1}_{x^k}x^{k+1}\rangle +\frac{L}{2}||\exp^{-1}_{x^k}x^{k+1}||^2.
\end{equation}
Thus, taking into account that $f(\cdot):=(g_1+g_2-h)(\cdot),$ inequalities (\ref{eqq1}), (\ref{eqq2}) and (\ref{eqq2a}), imply that
%\begin{equation}\label{eqq3}
\begin{eqnarray}\label{eq:decr:f}
    f(x^{k+1})-f(x^k) &\leq&\frac{\lambda_k}{2}d^2(x^{k},z^k)-\frac{\lambda_k}{2}d^2(x^{k+1},z^k)-\langle p^k,\exp^{-1}_{x^k}x^{k+1}\rangle\nonumber\\
&+&\frac{L}{2}d^2(x^{k+1},x^k),    
\end{eqnarray}
%\end{equation}
where $p^k=\lambda_k \exp_{x^k}^{-1} z^k=w^k-v^k.$ 
%Now, we consider the geodesic triangle $\triangle (z^k, x^k,x^{k+1})$ and set $\theta=\angle (\exp_{x^k}^{-1} z^k,\exp^{-1}_{x^k}x^{k+1}).$
Then, from \eqref{eq:tri} with $x=x^k,\,y=x^{k+1}$ and $z=z^k$ that
$$
d^2(z^k,x^k)+d^2(x^k,x^{k+1}) -2\langle \exp_{x^k}^{-1} z^k,\exp^{-1}_{x^k}x^{k+1} \rangle \leq d^2(z^k,x^{k+1}). 
$$
As a consequence,
\begin{equation}\label{eq:tri:conseq}
\frac{\lambda_k}{2}d^2(z^k,x^k) -\frac{\lambda_k}{2}d^2(z^k,x^{k+1})-\lambda_k\langle \exp_{x^k}^{-1} z^k,\exp^{-1}_{x^k}x^{k+1} \rangle \leq -\frac{\lambda_k}{2}d^2(x^k,x^{k+1}).
\end{equation}
Therefore, \eqref{eq:decr:f} and \eqref{eq:tri:conseq} imply that
$$
		f({x}^{k+1})\leq f(x^k)-\bigg(\frac{\lambda_k-L}{2}\bigg)d^2(x^{k+1},x^k), 
$$
from which the conclusion of the proof follows.

\end{proof}

As a consequence of the previous result, we next show that a sequence generated by R-PPM is monotonically decreasing.
% From now on, for the convergence analysis of R-PPM we assume the following assumption.

\begin{proposition}\label{pro2} Let $\{x^k\}$ be a sequence generated by  R-PPM. Then, 
\begin{enumerate}
\item[(a)] The sequence $\{f({{x}}^k)\}$ is monotonically decreasing. As a consequence, it is convergent;
\item[(b)]$\displaystyle\sum_{k=0}^{\infty} d^2(x^{k+1},x^k)<\infty$. In particular, $x^{k+1}-x^k \to 0,$ as $k\to+\infty$.
\end{enumerate}
\end{proposition}
\begin{proof}
$(a)$ It follows from Proposition \ref{pro1} that 
$$
f(x^{k+1})\leq f(x^k)-\bigg(\frac{\lambda_k-L}{2}\bigg)d^2(x^{k+1},x^k).
$$ Since $x^{k+1} \neq x^k$ and $\lambda_k - L>0$, then we have $f(x^{k+1})< f(x^k),$ for all $k\ge 0.$ Being $f$ bounded from below, see Assumption (A2), this implies that $\{f(x^k)\}$ converges.\\
$(b)$~It follows from Proposition \ref{pro1} and $0<\alpha:=\lambda_0-L\leq \lambda_k-L $ that
\begin{equation*}\label{ineq4}
\frac{\alpha}{2} d^2(x^{k+1},x^k)\leq {f({{x}}^k)-f({{x}}^{k+1})}. 
\end{equation*}
Summing up the previous inequality from $k=0$ to $k=N$, we have
$$
\sum_{k=0}^{N}d^2(x^{k+1},x^k)\leq \frac{2}{\alpha}(f({{x}}^0)-f({{x}}^{N+1})). 
$$
Letting $N$ goes to $+\infty$ in the above inequality, taking in account Assumption (A2), we obtain the two conclusions. \qed
\end{proof}

The next result shows that if there exists a cluster point of a sequence generated by R-PPM, then it is a critical point as defined in \eqref{critico}.

\begin{theorem}\label{mai:result1} Let $\{x^k\}$ be a sequence generated by  method R-PPM. Then, every cluster point of $\{x^k\}$ is a critical point of $f$.
\end{theorem}

\begin{proof} Let $x^*$ be a cluster point of $\{x^k\}$, and let $\{x^{k_j}\}$ be a subsequence of $\{x^k\}$ converging to $x^*$. Then, it follows from (\ref{wk}) that there exist bounded subsequences $\{w^{k_j}\}, \{v^{k_j}\}$ and $\{z^{k_j}\}$ of $\{w^{k}\}, \{v^{k}\}$ and $\{z^{k}\}$, respectively. Thus, we can suppose that $ w^{k_j} \to w^*, v^{k_j} \to v^*, z^{k_j} \to z^*$ (one can extract others subsequences if necessary). 

Now, it follows from (\ref{xk}) that $\displaystyle f_{{k_j+1}}(x^{k_j+1}) \leq f_{{k_j+1}}(x)$ for all $x \in \mathbb{R}^n$. In particular, for $x = x^*$ we obtain
$$
g_1(x^{k_j+1})\leq g_1(x^{*}) -\frac{\lambda_{k_j+1}}{2}d^2(x^{k_j+1},z^{k_j})+ \frac{\lambda_{k_j}}{2}d^2(x^*,z^{k_j}).
$$
From Proposition \ref{pro2}~(b) we have that $(x^{k_j+1}- x^{k_j})\to 0.$ As a consequence, we can guarantee  that~$x^{k_j+1}\to x^*$ as $j\to+\infty$. Since $\{\lambda_{k}\}$ convergent, thus the previous inequality implies that $\displaystyle\limsup_{j\to+\infty}g_1 (x^{k_j+1})\leq g_1(x^*).$
Combining this with the assumption that $g_1$ is lower semicontinuous, we obtain that
$$
\lim_{j\to+\infty}g_1(x^{k_j+1})= g_1(x^*).
$$
The same argument as in \eqref{eq:cons:opt} shows that there exists $u^{k_j+1} \in  \partial^L g_1(x^{k_j+1})$ satisfying $
u^{k_j+1}=\lambda_{k_j}\exp_{x^{k_j+1}}^{-1} (z^{k_j}).$ On the other hand $\lambda_{k_j} \exp_{x^{k_j}+1}^{-1} (z^{k_j})= w^{k_j}-v^{k_j}$ and therefore, we have that 
$$ 
u^*:=\lim_{j\to+\infty}u^{k_j+1}=w^*-v^*.
$$
 Thus, $x^{k_j+1} \to x^*, u^{k_j+1} \in  \partial^L g_1 (x^{k_j+1}), u^{k_j+1} \to u^*$ as $j\to+\infty$, it follows from the property of limiting subdifferentials that $u^{*} \in  \partial^L g_1(x^*)$. Therefore
 $$
 \bigg[\partial^L g_1(x^{*})+\grad g_2(x^*)\bigg]\cap \partial h(x^*)\neq \emptyset.
 $$
This means that $x^*$ is a critical point of $f$. Finishing the proof of the Theorem~\ref{mai:result1}.
\end{proof}

The next theorem establishes an upper bound on the iteration complexity of the R-PPM method required to obtain an approximate solution with a prescribed tolerance.
\begin{theorem}
	\label{complex1}
	Given $\epsilon > 0$, let $\{x^k\}_{k=0}^{T}$ be the sequence generated by R-PPM, for some $T\ge 1$.
    Then, the number of iterations such that
\begin{equation}\label{sup010}
d(x^{k+1}, x^k) > \epsilon \quad \text{and} \quad f(x^{k+1}) > f^{\text{low}}, \quad \text{for } k = 0, \ldots, T,
\end{equation}
is bounded from above by
\begin{equation*}
T < \frac{2(f(x^0) - f^{\text{low}})}{(\lambda_0-L) \epsilon^2}.
\label{eq:3.111}
\end{equation*}

\end{theorem}
\begin{proof}
By Proposition \ref{pro1} and $\lambda_0>L$, we have that 
$$
f(x^{k+1})\leq f(x^k)-\bigg(\frac{\lambda_0-L}{2}\bigg)d^2(x^{k+1},x^k).
$$ 
Using \eqref{sup010} in the last inequality, we have
\begin{equation}\label{0.00}    
f^{\text{low}}\leq f(x^{k+1})\leq f(x^k)- \frac{\lambda_0-L}{2}\epsilon^2, \quad \,\,k=0,1,\ldots, T.
\end{equation}
Then, summing up inequality  \eqref{0.00} from $k=0$ to $k=T$ and using $f^{low}\leq f(x^{k+1})$ we conclude that
$$f^{low}\leq f(x^{T+1})< f(x^0)- T\frac{\lambda_0-L}{2}\epsilon^2,
$$
this implies that 
$$T<\frac{2(f(x^0)-f^{low})}{(\lambda_0-L)\epsilon^2},
$$
which concludes the proof. 
\end{proof}

We conclude this section with some remarks. 
Assume that the tolerance parameter $\epsilon>0$ satisfies $1/\epsilon \leq \alpha$, which is possible since $\alpha\in \mathbb{R}_{++}$ is arbitrarily chosen in R-PPM. 
Then the following sequence is a valid example satisfying 
$0<L<\lambda_k\leq \lambda_{k+1}\leq L+\alpha$ with $\lambda_0-L=\epsilon^{-1}$:
\begin{equation}\label{def:lambda:example}
    \lambda_k \;=\; L + \frac{1}{\epsilon} \;+\; \Big(\alpha - \frac{1}{\epsilon}\Big)\bigl(1-q^{\,k}\bigr),
    \qquad k=0,1,2,\ldots,
\end{equation}
where $q\in(0,1)$ is a fixed parameter. 
Clearly, $\lambda_k$ is nondecreasing, 
$\lambda_0 = L + \epsilon^{-1}$, and $\lim_{k\to\infty}\lambda_k = L+\alpha$. 
Hence, if we choose $\{\lambda_k\}$ as in  \eqref{def:lambda:example}, it follows from Theorem \ref{complex1} that the complexity of R-PPM is  ${\cal O}(\epsilon^{-1}).$

\section{Adaptive R-PPM}\label{adaptive}

As noted in the previous section, to implement R-PPM is necessary to know the constant $L$, due to the condition $\lambda_k \in (L, L+\alpha],$ for some $\alpha>0$ given, see its {\bf Step 0}. 
However, in practice, this constant may not be readily available or may require a greater effort to access it. To address this difficulty, we propose in this section a method similar to R-PPM but adaptive. 
That is, we do not need to know $L$ or any other constant that may be difficult to compute. 
% Moreover, the variant of R-PPM proposed in this section allows to solve its subproblem inexactly. 

\begin{lemma}\label{lem:precond}
	Let $x, z\in {\cal M}$ and $\lambda>0$ be such that 
 $v:=\grad g_2(x)$ and 
 	\begin{equation}\label{eq:zwv}
 z:=\exp_{x}\bigg(\frac{1}{\lambda}(w-v)\bigg),
 \end{equation}
 with $w\in \partial h(x).$ 
 Suppose that $y\in {\cal M}$ is a solution to the problem
 \begin{equation}\label{aux:problem:line:search}
 \min_{x\in \cal{M}} \left\{g_1(x)+\frac{\lambda}{2}d^2(x,z)\right\},
 \end{equation}
 % satisfies
	% \begin{equation}\label{eq:opt}
	% 	g_1(y)+\frac{\lambda}{2}\left[d^2(y,z)-d^2(x,z)\right]\leq g_1(x).
	% \end{equation}
	%  Moreover, 
    and assume that
	\begin{equation}\label{eq:lambda}
		\lambda\ge 2L,
	\end{equation}	
    where $L$ is the Lipschitz constant to $\grad g_2(\cdot)$.
	Then,
	\begin{equation}\label{eq:desce}
		f(x)-f(y)\ge \frac{\lambda}{4}d^2(y,x).
	\end{equation}	
\end{lemma}
\begin{proof}
We first observe that the assumption that $h$ is convex implies that
	\begin{equation*}\label{eqq22}
		\langle w,\exp^{-1}_{x}y\rangle - h(y) \leq  - h(x),
	\end{equation*}
	with $w\in \partial h(x).$	 
Since $\grad g_2(\cdot)$ is $L$-Lipschitz continuous, by assumption, we can apply Proposition \ref{prop5} with  $g_2=\psi,$  $v=\exp^{-1}_{x}y$ and $p=x$, combined with the assumption that $v=\grad g_2(x)$, to conclude that
	\begin{equation*}\label{eqq2a2}
		g_2(y) \leq g_2(x)+ \langle v,\exp^{-1}_{x}y\rangle +\frac{L}{2}||\exp^{-1}_{x}y||^2.
	\end{equation*}
The two previous inequalities combined with the assumption that $y$ is a solution of Problem \eqref{aux:problem:line:search} yields
\begin{eqnarray*}
g_1(y)+g_2(y)-h(y)&\le& g_1(y)+g_2(x)+ \langle v-w,\exp^{-1}_{x}y\rangle\\ &+&\frac{L}{2}||\exp^{-1}_{x}y||^2-h(x)\\
&\le& g_1(x)+g_2(x)-h(x)+\frac{\lambda}{2}d^2(x,z)-\frac{\lambda}{2}d^2(y,z)\\
&-&\langle w-v,\exp^{-1}_{x}y\rangle +\frac{L}{2}||\exp^{-1}_{x}y||^2.
\end{eqnarray*}
Using now that $f(\cdot)=(g_1+g_2-h)(\cdot)$ and the fact that $\lambda \exp_{x}^{-1} z=w-v,$ in view of \eqref{eq:zwv}, the above inequality becomes to
\begin{eqnarray}\label{eq:qs:mono}
f(y)-f(x)  \leq 
\frac{\lambda}{2}\left(d^2(x,z)-d^2(y,z)\right)-\lambda\langle \exp_{x}^{-1} z,\exp^{-1}_{x}y\rangle+\frac{L}{2}d^2(y,x).
	\end{eqnarray}
	%\end{equation}
On the other hand, the inequality  \eqref{eq:tri} implies that
     %consider the geodesic triangle $\triangle (z, x,y)$ and set $\theta=\angle (\exp_{x}^{-1} z,\exp^{-1}_{x}y)$. It follows from  Theorem \ref{triangle} that
% $$
% d^2(x,z)+d^2(x,y) -2\langle \exp_{x}^{-1} z,\exp^{-1}_{x}y \rangle \leq d^2(y,z), 
% $$
% 	implying that
	$$
	\frac{\lambda}{2}d^2(x,z) -\frac{\lambda}{2}d^2(z,y)-\lambda\langle \exp_{x}^{-1} z,\exp^{-1}_{x}y \rangle \leq -\frac{\lambda}{2}d^2(x,y).
	$$
	Thus, the previous inequality and \eqref{eq:qs:mono} imply in
	$$
	f(y)\leq f(x)-\bigg(\frac{\lambda-L}{2}\bigg)d^2(y,x). 
	$$
	Since \eqref{eq:lambda} holds, then we can conclude from the previous inequality that \eqref{eq:desce} holds,  which concludes the proof.
	% \[
	% f(y)\leq f(x)-\frac{\lambda}{4}d^2(y,x),
	% \]
\end{proof}

Taking into account the preivous result, we now state an adaptive proximal-point algorithm to solve the Problem \eqref{dcg}.

\noindent

	\hrule
	\vspace{0.2cm}
	\noindent
	{\bf Adap-RPPM}
	\vspace{0.1cm}
	\hrule
\begin{itemize}
    \item [ ] {\bf Step~0.}
 Given an initial point $x^0 \in \dom f,$ $\lambda_0>0,$ $\epsilon>0$ and $k:=0$ 
%\item [ ] {\bf Step~1.} Find the smallest integer $i\ge 0$ such that $2^i\lambda_k\ge 2\lambda_0.$
\item [ ] {\bf Step~1.} Given $x^k \in \dom f,$ calculate 
\begin{equation*}\label{vwz}
	v^k:=\grad\,g_2(x^k)~ \mbox{and}~ w^k \in \partial h(x^k).
    \end{equation*}
\item[] {\bf Step~1.2} Calculate 
\[
z^k=\exp_{x^k}{\bigg(\frac{1}{\lambda_k}(w^k-v^k)\bigg)}. 
\]
\item [ ] {\bf Step~2.} Compute 
\begin{equation*}\label{yk}
	 x^{k+1}\in \min_{x\in \cal{M}} \left\{\tilde f_k(x):=g_1(x)+\frac{\lambda_k}{2}d^2(x,z^k)\right\}.
\end{equation*}
% such that
% \begin{equation}\label{test:adap}
%     g_1(y^{k+1})+\frac{\lambda_k}{2}\left[d^2(y^{k+1},z^k)-d^2(x^k,z^k)\right]\leq g_1(x^k).
% \end{equation}

\item [ ] {\bf Step~3.} If $d(x^{k+1}, x^k)\le \epsilon$, stop declaring that $x^{k+1}$ is a solution of \eqref{dcg}. 
Otherwise, go to {\bf Step 4}.

\item [ ] {\bf Step~4.} If 
\begin{equation}\label{def:mon}
f(x^{k+1})-f(x^k)\le -\frac{\lambda_k}{4}d^2(x^{k+1},x^k),
\end{equation}	
then 
% $x^{k+1}=y^{k+1},$ 
$\lambda_{k+1}\gets \lambda_k$, $k\gets k+1$, and go to {\bf Step 1}.
Otherwise, $\lambda_k\gets 2\lambda_k$, and go to {\bf Step 1.2}.
\end{itemize}

We now make some remarks about Adap-RPPM. 
Initially, Adap-RPPM requires the parameters $x^0$, $\lambda_0$, and $\epsilon$, which are provided by the user and can be freely chosen. 
According to \eqref{def:mon}, small values of $ \lambda_k$ tend to produce larger decreases, which, hopefully, makes the method converge more quickly. 
It is important to emphasize that, for its execution, Adap-RPPM does not require prior knowledge of the constant $L$ from the Lipschitz condition of the gradient of $g_2$, which makes the method advantageous in problems where this constant is unknown.

The well-defined Step $2$ of the Adap-RPPM is ensured, in particular, by item~a) of Proposition~\ref{prop:well:defined}.
In turn, the well-defined Step $4$ follows directly from Lemma~\ref{lem:precond}.
Based on these results, we conclude that Adap-RPPM is well-defined.

Next, we demonstrate that the sequence $\{\lambda_k\}$ generated by Adap-RPPM is bounded, a fundamental result to ensure the convergence of the method.

\begin{lemma}\label{eq:bound.sigmak0}
	Suppose that $\grad g_2$ is $L$-Lipschitz continuous on ${\cal M}.$ Then, the sequence $\{\lambda_k\}$ in Adap-RPPM satisfies
	\begin{equation}\label{eq:bound.sigmak}
		 \lambda_0 \leq \lambda_k\leq 4L+\lambda_0=:\lambda_{\max}.
	\end{equation}
    
\end{lemma}	
\begin{proof} By Lemma \ref{lem:precond}, we have that the decrease in Step 4 is satisfied whenever $\lambda_k \geq 2L$. Clearly, \eqref{eq:bound.sigmak} is true if $\lambda_0 \geq 2L$, since in this case $\lambda_k$ does not need to be updated and will remain constant, equal to $\lambda_0$.

On the other hand, if $\lambda_0 < 2L$ and the parameter is updated so that there exists $m$ such that $\lambda_m \leq 2L < \lambda_{m+1}$, then we have that $2L < 2\lambda_m = \lambda_{m+1} \leq 4L$. Since the parameter is not updated when it is greater than or equal to $2L$, we conclude that, starting from $2L \leq \lambda_{m+1} \leq 4L$, the sequence of parameters remains constant and equal to $\lambda_{m+1}$.

Therefore, we can conclude that the sequence $\{\lambda_k\}$ is bounded from above by $4L + \lambda_0$.
\end{proof}

The following theorem establishes an iteration-complexity bound of $\mathcal{O}\left(\epsilon^{-2}\right)$ for the proposed method.
\begin{theorem}
	\label{thm:3.1}
	Given $\epsilon>0$, let $\left\{x^{k}\right\}_{k=1}^{T}$ be generated by the Adap-RPPM, such that
	\begin{equation*}
		d(x^{k+1}, x^k)>\epsilon,\quad k=0,\ldots,T.
		\label{eq:3.10}
	\end{equation*}
	Then,
	\begin{equation*}
		T<\frac{4(f(x^0)-f^{low})}{\lambda_0\epsilon^2}.
		\label{eq:3.11}
	\end{equation*}
\end{theorem}
\begin{proof}
We first note that by {\bf Step 4}, $d(x^{k+1}, x^k)>\epsilon$, for each $k=0,\ldots,T$. 
Combining  this fact, \eqref{def:mon} and \eqref{eq:bound.sigmak}, we have 
\begin{equation}\label{0.0}    
f(x^{k+1})\leq f(x^k)- \frac{\lambda_k}{4}\epsilon^2\leq f(x^k)- \frac{\lambda_0}{4}\epsilon^2, \,\,\forall k=0,1,\ldots, T.
\end{equation}
Then, summing up inequality  \eqref{0.0} from $k=0$ to $k=T$, and using Assumption (A2), we conclude that
$$f^{\text{low}}\leq f(x^{T+1})< f(x^0)- T\frac{\lambda_0}{4}\epsilon^2,
$$
this implies that 
$$T<\frac{4(f(x^0)-f^{\text{low}})}{\lambda_0\epsilon^2},
$$
which concludes the proof.
\end{proof}

\section{Under the Kurdyka-Lojasiewicz Property}\label{kl}

In the context of descent algorithms applied to convex functions, one typically ensures global convergence, meaning that the entire sequence generated by the algorithm converges.
However, when the objective functions lack convexity or quasiconvexity, the resulting sequences may display significant oscillations, and only partial convergence can generally be established. 
To address such challenges, the Kurdyka–Lojasiewicz (KL) property has emerged as a powerful analytical tool, especially in the study of asymptotic behavior of algorithms, such as the proximal point method.

Our main convergence result in this section is restricted to functions that satisfy the KL property, as stated below.
\begin{definition} \label{def:KL}
 A lower semicontinuous function $f:{\cal M}\to  \mathbb{R} \cup \{ +\infty\}$ satisfies the  KL property  at  $x^* \in \text{dom}\, \partial^L f$ if there exist $\nu > 0$, a neighborhood $\mathbb{V}$ of  $x^*$, and a continuous
concave function $\psi: [0, \nu[ \to [0, \infty[$ with
\begin{enumerate}
\item[(a)]$\psi(0)=0$;
\item[(b)] $\psi$ is continuously differentiable on $]0, \nu[$;
\item[(c)] $\psi^{\prime} >0$ on $]0, \nu[$;
\item[(d)] For every $x \in \mathbb{V}$ with $f(x^*)<f(x)< f(x^*)+ \nu$, we have
\begin{equation*}\label{KL}
\psi^{\prime}( f(x)- f(x^*))\, d(0, \partial^L f(x))\geq 1.
\end{equation*}
\end{enumerate}
\noindent
\end{definition}

% In our convergence analysis, we rely on the following assumption:
% \begin{itemize}
% \item[(A3)]\label{(B2)}
% There is a constant $\bar{L}>0$ satisfying 
%     % \textcolor{blue}{Já apareceu antes ou a primeira vez aqui? Fazer citações}	
% \begin{equation}\label{lim} 
%  \|\exp_{x^{k}}^{-1} z^{k}- \exp_{x^{k}}^{-1} z^{k-1}\| \leq \bar{L} d(x^k,x^{k-1}),
% \end{equation}
% where $z^{k}$ and $x^{k}$ are sequences generated by R-PPM.
% \end{itemize}

% \begin{remark} If ${\cal M}=\mathbb{R}^n$, $g_2$ and $h$ are differentiable with gradient Lipschitz continuous then the  Assumption \eqref{(B2)} holds.
% % \textcolor{blue}{Entendi. Mas precisa escrever com muito mais detalhes} 
% \end{remark}

% Case 1: ${\cal M}=\mathbb{R}^n$
% \begin{align*}
%  \|\exp_{x^{k}}^{-1} z^{k}&- \exp_{x^{k}}^{-1} z^{k-1}\|=\|(z^{k}-x^k)- (z^{k-1}-x^{k})\| \\
%  &=\|z^{k}-z^{k-1}\|=\left\|x_k+\frac{1}{\lambda}(w^k-v^k)-x_{k-1}-\frac{1}{\lambda}(w^{k-1}-v^{k-1})\right\|\\
%  &\le\|x_k-x_{k-1}\| +\frac{1}{\lambda}\|w^k-w^{k-1}\|+\frac{1}{\lambda}\|v^k-v^{k-1}\|
% \end{align*}

Next, we present an auxiliary result regarding the sequence generated by R-PPM. It is important to prove the main result in this section.

\begin{lemma}\label{lem:auxiliar}
Let $\{x^k\}$, $\{z^k\}$, $\{v^k\}$, and $\{w^k\}$ be sequences generated by R-PPM. Suppose that $\lambda_k=\lambda>0$, for every $k$, 
and that the function $h$ is differentiable with a continuous Lipschitz gradient with the associated constant $\tilde M_1\ge 0$. 
% Moreover, assume that 
If $\{x^k \}$ has a cluster point 
then, for $k$ sufficiently larger, there exists $\tilde M>0$ such that
\[
\|\exp_{x^{k}}^{-1} z^{k}- \exp_{x^{k}}^{-1} z^{k-1}\|\le \tilde M d(x^k,x^{k-1}).
\]
\end{lemma}
%Case 2: ${\cal M}$
% Assume that there exists $r_x>0$ such that $\exp_x: B(0,r_x)\subset T_x{\cal M} \to {\cal M}$ is a difeomorfism. 
% Then,
\begin{proof} We first observe that if $\{x^k\}$, $\{z^k\}$, $\{v^k\}$, and $\{w^k\}$ are sequences generated by R-PPM, and $\lambda_k=\lambda>0$, for every $k$, then it follows from \eqref{wk} that
\[
\exp_{x^{k}}^{-1} z^{k}=\frac{1}{\lambda}(w^k-v^k).
\]
Moreover, using the fact that $g_2$, $h$ and $\exp$ are continuous, and the assumption that $\{x^k \}$ has a cluster point, we can conclude that $\{z^k\}$ has a cluster point, and hence that there exists $\tilde M_2\ge 0$ such that, for $k$ sufficiently larger, there holds
\begin{equation*}
\left\|\Gamma_{x^k,x^{k-1}}\exp_{x^{k-1}}^{-1} z^{k-1}-\exp_{x^{k}}^{-1} z^{k-1}\right\|\le \tilde M_2  d(x^k,x^{k-1}).  
\end{equation*}
Using these facts and simple algebraic manipulations we obtain that
\begin{align*}
 &\|\exp_{x^{k}}^{-1} z^{k}- \exp_{x^{k}}^{-1} z^{k-1}\| =\left\|\frac{1}{\lambda}(w^k-v^k)-\exp_{x^{k}}^{-1} z^{k-1}\right\|\\
&\le \left\|\frac{1}{\lambda}(w^k-v^k)-\Gamma_{x^k,x^{k-1}}\exp_{x^{k-1}}^{-1} z^{k-1}\right\|\\
&+\left\|\Gamma_{x^k,x^{k-1}}\exp_{x^{k-1}}^{-1} z^{k-1}-\exp_{x^{k}}^{-1} z^{k-1}\right\|\\
&=\left\|\frac{1}{\lambda}(w^k-v^k)-\frac{1}{\lambda}\Gamma_{x^k,x^{k-1}}(w^{k-1}-v^{k-1})\right\|\\
&+\left\|\Gamma_{x^k,x^{k-1}}\exp_{x^{k-1}}^{-1} z^{k-1}-\exp_{x^{k}}^{-1} z^{k-1}\right\|\\
&\le \frac{1}{\lambda}\left\|w^k-\Gamma_{x^k,x^{k-1}}w^{k-1}\right\|+\frac{1}{\lambda}\left\|v^k-\Gamma_{x^k,x^{k-1}}v^{k-1}\right\|+\tilde M_2d(x^k,x^{k-1})\\
&\le \left(\frac{\tilde M_1+L}{\lambda}+\tilde M_2\right)d(x^k,x^{k-1})=\tilde M d(x^k,x^{k-1}), 
\end{align*} 
which concludes the proof. 
\end{proof}

% {\bf Vc não acha que é mais fácil assumir/argumentar que
% \[
% \left\|\Gamma_{x^k,x^{k-1}}\exp_{x^{k-1}}^{-1} z^{k-1}-\exp_{x^{k}}^{-1} z^{k-1}\right\|\le \tilde M_2d(x^k,x^{k-1})
% \]
% do que a anterior?}

In the following theorem, we establish sufficient conditions that guarantee the convergence of the
sequence $\{x^k \}$ generated by the R-PPM. 
For this, we will consider that $f$ satisfies the KL property and that $\grad h$ is $\tilde M_1$-Lipschitz continuous.
% for some positive constant $\tilde M_1>0.$

\begin{theorem}\label{TeoKL} Let $\{x^k\}$ be a sequence generated by R-PPM, and assume that $\lambda_k=\lambda>0$, for every $k$. 
Suppose that  
% Assumption~\eqref{(B2)} holds,
$\grad h$ is $\tilde M_1$-Lipschitz continuous and that $f$ has the KL property at any point $x \in \dom\,f$. 
If the sequence $\{x^k \}$ has a cluster point $x^*$, then the whole sequence converges to $x^*$, which is a critical point of $f$.
\end{theorem}
\begin{proof} Since $\{x^k \}$ has a cluster point $x^*$, by assumption, we can select  a subsequence $\{x^{k_j}\}$ that converges to $x^*$. By Proposition \ref{pro2}~$(a)$, without loss of generality, we can assume that $f(x^{k})> f(x^{*})$ for all $k$, otherwise, there exists $k$ such that $x^{k+1}=x^k$, i.e, the method stop after a finite number of steps. We have
\begin{align}\label{eq:limiting}
\partial^L f(x^{k}) = \partial^L g_1(x^{k})+ v^{k}-w^{k}=\partial^L g_1(x^{k}) -\lambda \exp_{x^{k}}^{-1} z^{k}
\end{align}
where $v^{k}=\grad~g_2(x^{k}),$ $w^{k}=\grad h(x^{k})$ and $w^k-v^k=\lambda \exp_{x^{k}}^{-1} z^{k},$ in view of \eqref{wk}.
Now, it follows from \eqref{xk} that
\begin{align}\label{eq:first:order:stationay}
0 &\in \partial^L g_1(x^{k})-\lambda \exp_{x^{k}}^{-1} z^{k-1}, \forall k\ge 1.
\end{align}
% holds for every $k\ge 1$. 
Using now Lemma \ref{lem:auxiliar}, \eqref{eq:limiting} and \eqref{eq:first:order:stationay}, we obtain that
\begin{align}\label{dis}
d(0,\partial^L f(x^{k}))\leq \lambda\| \exp_{x^{k}}^{-1} z^{k-1}- \exp_{x^{k}}^{-1} z^{k}\|
% &\le 2|\lambda_{k-1}-\lambda_{k}|\|\exp_{x^{k}}^{-1} z^{k-1}-\exp_{x^{k}}^{-1} z^{k}\|\nonumber\\
% &\le 4(L+\alpha)\| \exp_{x^{k}}^{-1} z^{k-1}- \exp_{x^{k}}^{-1} z^{k}\| \nonumber\\
\leq \lambda \tilde M d(x^k,x^{k-1}).
\end{align}
% where the third inequality above follows from Step 0 of PPMNN. 

Let $\delta >0$  be such that $\mathbb{B}(x^*;\delta) \subset \mathbb{V}$, where $\mathbb{V}$ is as in Definition~\ref{def:KL}. As  $x^{k_j} \to x^*, 
\|x^{k+1}-x^k\| \to 0, f(x^{k}) \to f(x^*)$, and $f(x^{k})> f(x^*)$ for all $k$, we can find a natural number $p$, large enough, satisfying
\begin{equation*}\label{est}
x^p \in \mathbb{B}(x^*;\delta),~ f(x^*)<f(x^p)< f(x^*)+ \nu
\end{equation*}
and
\begin{equation}\label{esti}
d(x^p,x^*)+ \frac{d(x^p,x^{p-1})}{4}+\gamma \psi\bigg(f(x^{p})- f(x^*)\bigg)< \frac{3\delta}{4}
\end{equation}
where $\gamma:=2\lambda\tilde M/(\lambda-L)$.
% and $(\alpha, L)$ is as in PPMNN and  $\bar{L}$ as in \eqref{lim}.

We now claim that for all $k\ge p$, $x^k \in \mathbb{B}(x^*,\delta)$.
To show the claim, we first show that whenever  $x^k \in \mathbb{B}(x^*,\delta)$ and $f(x^*)<f(x^k)< f(x^*)+ \nu$ for some $k$, we have 
\begin{equation}\label{ineq5}
 d(x^{k+1},x^{k})\leq  \frac{d(x^{k},x^{k-1})}{4} +\gamma \bigg[\psi\bigg(f(x^{k})- f(x^*)\bigg)-\psi\bigg(f(x^{k+1})- f(x^*)\bigg) \bigg]. 
\end{equation}
Indeed, due to the property that $\{f({{x}}^k)\}$ is monotonically decreasing (see Proposition~\ref{pro1}) and the assumption that $\psi^{\prime} >0$ on $]0, \nu[$ (see Definition~\ref{def:KL}(c)), we conclude from \eqref{dis} that
\begin{align*}
\lambda\tilde M d(x^{k},x^{k-1})\bigg[\psi\bigg(f(x^{k})- f(x^*)\bigg)-\psi\bigg(f(x^{k+1})- f(x^*)\bigg) \bigg]  \\
\geq d(0,\partial^L f(x^{k})) \bigg[\psi\bigg(f(x^{k})- f(x^*)\bigg)-\psi\bigg(f(x^{k+1})- f(x^*)\bigg) \bigg].
\end{align*}
As a consequence, we use the concavity of $\psi$, and combine the previous inequality with \eqref{KL}, and \eqref{lem:mon:decre} to conclude that
\begin{eqnarray*}
&&\lambda\tilde Md(x^{k},x^{k-1})\bigg[\psi\bigg(f(x^{k})- f(x^*)\bigg)-\psi\bigg(f(x^{k+1})- f(x^*)\bigg) \bigg]  \\
 &\geq& d(0,\partial^L f(x^{k})) \psi^{\prime}\bigg(f(x^{k})- f(x^*)\bigg)\bigg(f(x^k)-f(x^{k+1})\bigg)\\
 &\geq& \frac{\lambda-L}{2} d^2(x^{k+1},x^{k}).
\end{eqnarray*}
Hence,
\begin{align*}
\gamma \bigg[\psi\bigg(f(x^{k})- f(x^*)\bigg)-\psi\bigg(f(x^{k+1})- f(x^*)\bigg) \bigg]  \geq \frac{d^2(x^{k+1},x^{k})}{d(x^{k},x^{k-1})}\\
\geq d(x^{k+1},x^{k})- \frac{d(x^{k},x^{k-1})}{4},
\end{align*}
where the last inequality holds since $a^2/b\geq a -b/4$, for any positive real numbers $a$ and $b$. As a consequence, \eqref{ineq5} holds.
%  d(x^{k+1},x^{k})\leq  \frac{d(x^{k},x^{k-1})}{4} +\gamma \bigg[\psi\bigg(f(x^{k})- f(x^*)\bigg)-\psi\bigg(f(x^{k+1})- f(x^*)\bigg) \bigg] 
% \end{equation}

We now prove the claim by induction on $k$. 
The claim is true for $k=p$ due to \eqref{esti}. We suppose that it also holds for $k = p,p+1,\cdots ,p+n-1$, with $n\geq 1$, i.e., $x^p, x^{p+1},\cdots, x^{p+n-1}  \in \mathbb{B}(x^*;\delta)$. Since $f(x^*) < f(x^{k+1})\leq f(x^k)$, our choice of $p$ implies that $f(x^*) < f(x^k)< f(x^*)+\nu$  for all $k \geq  p$. In particular, \eqref{ineq5} can be applied for all $k = p,p+1,\cdots ,p + n-1$. Thus, summing (\ref{ineq5}) from $k=p$ to $k=p+n-1$, doing some algebraic manipulations and using the fact that $\psi\geq 0$, we obtain 
\begin{equation}\label{ineq6}
\sum_{j=1}^nd(x^{p+j},x^{p+j-1})\leq \frac{4}{3}\bigg[\frac{d(x^{p},x^{p-1})}{4}+\gamma \bigg[\psi\bigg(f(x^{p})- f(x^*)\bigg)  \bigg].
\end{equation} 
As a consequence, we have
\begin{align*}
d(x^{p+n},x^{*}) &\leq d(x^{p},x^{*})+ \sum_{j=1}^nd(x^{p+j},x^{p+j-1})\\
&\leq \frac{4}{3}\bigg[d(x^{p},x^{*})+\frac{d(x^{p},x^{p-1})}{4}+\gamma \bigg[\psi\bigg(f(x^{p})- f(x^*)\bigg)  \bigg]<\delta,
\end{align*}
where the last two inequalities follows from \eqref{ineq6} and \eqref{esti}, respectively. Thus, $x^k \in \mathbb{B}(x^*;\delta)$ and $f(x^*) < f(x^k)< f(x^*)+\nu$  for all $k \geq  p$. Using  \eqref{ineq6} we conclude that  $\sum_{k=1}^{+\infty} d(x^{k+1},x^{k})<+ \infty$. Therefore, $\{x^k\}$ is a Cauchy sequence and, hence, is convergent.
\end{proof}

% {\color{red}
% \section{Applications-Depois ver onde fica mais adequado colocar isso}\label{applications}

% }

\section{Application and Numerical Experiments}\label{secexp}
This section contains two subsections. The first presents a practical example of \eqref{dcg}.
The second is dedicated to numerical experiments.

\subsection{A Practical Example of \eqref{dcg}}
In this subsection, we show how the log-determinant optimization problem with constraints can be rewritten as in \eqref{dcg}.

When $g_1$ is the indicator function of a convex set $\Omega\subset {\cal M}$, \eqref{dcg} is equivalent to solving the following problem 
\begin{equation}\label{dcgapp}
	\min_{x \in \Omega}\{g_2(x)-h(x)\}.
\end{equation}
In fact, \eqref{dcgapp} is equivalent to solving the following problem:
\begin{equation}\label{dcgapp1}
	\min_{x \in {\cal M}}\{g_1(x)+g_2(x)-h(x)\}
\end{equation}
where $ g_1(x) = \mathcal{I}_\Omega(x)$, with $ \mathcal{I}_\Omega $ being the indicator function defined as follows:
$$
{\cal I}_\Omega(x) = 
\begin{cases}
0, & \text{if } x \in \Omega \\
+\infty, & \text{if } x \notin \Omega
\end{cases}.
$$
We now consider the following particular case of \eqref{dcgapp}
% {\bf Log-determinant optimization problem - }
    \begin{equation} \label{ld1}
\min_{X \in \Omega} \left\{ -\log \det(X + \Sigma_1) + \lambda \log \det(X + \Sigma_2)\right\}
\end{equation}
where the variable is $X \in \mathbb{S}^n$, and the coefficients are 
$\Sigma_1, \Sigma_2 \in \mathbb{S}^n_{++}$, $C \in \mathbb{S}^n_{+}$,  $\lambda > 1$ and $\Omega=\{ X \in \mathbb{S}^n \mid 0 \preceq X \preceq C \}$.
This problem originates from a broad range of applications, particularly in network information theory and wireless communication; see \cite{liu2007extremal,weingarten2006capacity}.

Hence, if we set $h(X): = \log \det(X + \Sigma_1)$, $g_2(X) := \lambda \log \det(X + \Sigma_2)$, and $g_1(X) := {\cal I}_\Omega(X)$, the problem in \eqref{ld1} is equivalent to the problem in \eqref{dcgapp1}, since $h$ is convex over the bounded convex set $\Omega$ and $g_2$ has continuously Lipschitz gradients in the interior of $\Omega$. 
Specifically, \eqref{ld1} is equivalent to
\[
\min_{X \in \mathbb{S}^n} \left\{{\cal I}_\Omega(X)  + \lambda \log \det(X + \Sigma_2)-\log \det(X + \Sigma_1)\right\}.
\]

% More details about the problem in \eqref{ld1} can be found in \cite{yao2023global}.

\subsection{Numerical Experiments}
In this subsection, we evaluate the numerical performance of the proposed Adap-RPPM method. 
We consider three illustrative minimization problems to assess the algorithm’s behavior in different scenarios. 
All experiments were conducted on a workstation equipped with a Xeon W3-2435 processor and 64 GB of RAM. 
The algorithms were implemented in Julia (version 1.11.5), using the Manopt toolbox \cite{Bergmann2022} to solve the subproblem in \eqref{yk} via the \texttt{trust\_regions} solver. 
The stopping criterion was set to $d(X^{k+1},X^k) < 10^{-8}$, with a maximum of 100 iterations. 
The complete source code used in the numerical experiments is freely available at \url{https://github.com/mbortoloti/AdapRPPM}. 

\subsubsection{Sensitivity to the Parameter $\lambda_0$}

We first consider the minimization of
\begin{equation*}
    f_1(X) = \alpha\,\mathrm{tr}(X) + \mathrm{tr}(X^{-1}A) + \log(\det(X)) - n  - \mathrm{tr}(BX),
\end{equation*}
on $\mathbb{P}^n_{++}$, where $0 < \alpha < 1$, 
% \textcolor{red}{$\mu > 0$}, 
and $A,B$ are symmetric positive definite matrices. 
Here, $\mathrm{tr}(\cdot)$ denotes the trace of a matrix. 
In the DC decomposition we identify 
\[
g_1(X) = \alpha\, \mathrm{tr}(X), \quad g_2(X) = \mathrm{tr}(X^{-1}A)+\log(\det(X))-n, \quad h(X)=\mathrm{tr}(BX).
\]
The Riemannian gradient of $g_2$ is Lipschitz continuous with constant $L=1.0$. 

For the experiments, we set $A = \mathrm{Diag}(1,\ldots,n)$ and $B = \mu A$ ($\mu > 0$). 
The unique critical point of $f_1$ is 
\[
X^* = \mathrm{Diag}(x_1^*, \ldots, x_n^*), \quad 
x_i^* = \frac{-1 + \sqrt{1+4(\alpha - \mu i)i}}{2(\alpha - \mu i)}, \quad i=1,\ldots,n.
\]
We initialized the algorithm with $X^0 = \log(n) I$, where $I$ is the identity matrix. 
The tests were carried out with $\alpha = 0.5$, $\mu = 2.0$, and $n = 10$, corresponding to $\dim \mathbb{P}^n_{++} = 55$. 
The method was executed until $d(X^k, X^{k+1}) < 10^{-8}$, with initial parameters $\lambda_0 \in \{10^{-4}, 10^{-3}, 10^{-2}, 10^{-1}, 1.0\}$.

Figure~\ref{example_01_COST} shows the evolution of $|f(X^k) - f(X^*)|$ as a function of the iteration count. 
The results confirm that $\lambda_0$ strongly influences the convergence speed. 
For $\lambda_0 = 1.0$, convergence is reached in 42 iterations, while for $\lambda_0 = 10^{-4}$ the method requires 56 iterations. 
The slowest case occurs with $\lambda_0 = 10^{-2}$, where 93 iterations are needed. 
This behavior can be explained by the role of $\lambda_0$ in controlling the step size: very small values make the initial steps overly conservative, while $\lambda_0 = 1.0$ yields more aggressive yet stable steps. 
Note that the sequence $\{\lambda_k\}$ remains constant throughout the iterations. 

\begin{figure} 
    \centering
    \includegraphics[scale=0.4]{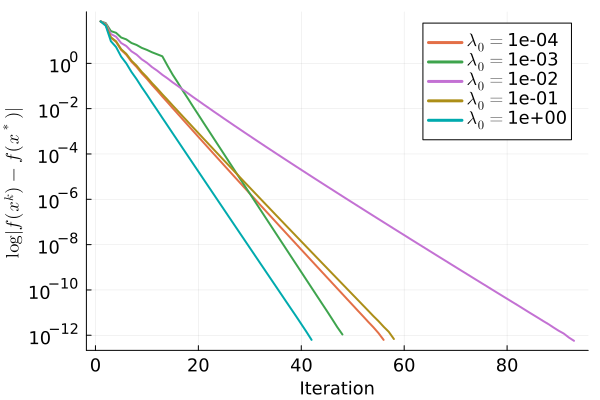}
    \caption{Comparison of Adap-RPPM and DCPPA in terms of CPU run time for the $f_1$ problem.}
\label{example_01_COST} 
\end{figure}

\begin{table}[h] 
\caption{Values of $\lambda_0$, $\lambda_k$, and iteration count for $f_1$.}
\centering 
\begin{tabular}{|c|ccccc|}
        \hline
         $\lambda_0$  & $10^{-4}$ & $10^{-3}$ & $10^{-2}$  & $10^{-1}$  &  $1.0$\\
         \hline
         $\lambda_k$  & 0.8192   & 1.024     & 0.64       &  0.8       &   1.0 \\
         \hline
         Iterations $k$ & 56       &   48      &  93        &   58       &  42 \\
        \hline
    \end{tabular} \label{tab1}
\end{table}

\subsubsection{Scalability with Manifold Dimension}

We next analyze Adap-RPPM for the function $f \colon \mathbb{P}^n_{++} \to \mathbb{R}$ defined by
\[
f(X) = \tfrac{1}{12} \bigl( \log \det X \bigr)^4 + \bigl( \log \det X \bigr)^3 - \log \det X.
\]
Its critical points satisfy $\det X^* \approx 1.756$. 

The experiments were conducted with matrix sizes $n = 10,15,\ldots,100$, corresponding to manifold dimensions ranging from 55 to 5050. 
The initialization was set to $X^0 = \log(n) I$ for each $n$, with $\lambda_0 = 10^{-4}$. 
Figure~\ref{example_02_dim_time} illustrates the computational cost, revealing a remarkably weak dependence of CPU time on the manifold dimension. 
The observed growth rate, benchmarked against a reference slope of $0.07\%$, highlights the scalability of the proposed method in high-dimensional settings. 

\begin{figure} 
    \centering
    \includegraphics[scale=0.4]{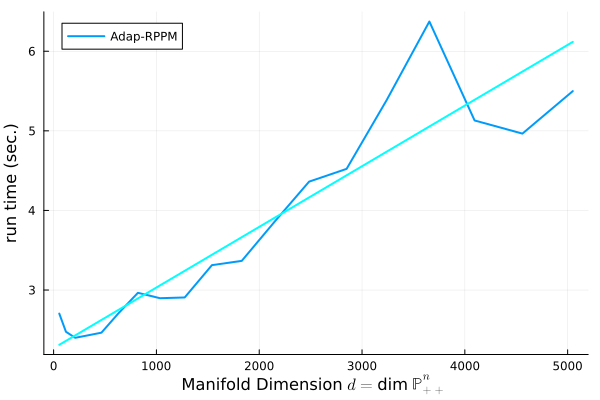}
    \caption{Comparison of Adap-RPPM and DCPPA in terms of CPU run time for the $f_2$ problem.}
\label{example_02_dim_time} 
\end{figure}

\subsection{Comparison with Baseline Methods}

Finally, we consider the minimization of
\[
f_3(X) = \bigl(\log(\det X)\bigr)^4 - \bigl(\log(\det X)\bigr)^2,
\]
defined on $\mathbb{P}^n_{++}$. 
In this case, the decomposition is $f_3 = g_1 - h$ with $g_1(X) = (\log\det X)^4$ and $h(X) = (\log\det X)^2$ (so $g_2 = 0$). 
The critical points are all $X^* \in \mathbb{P}^n_{++}$ satisfying $\det X^* = 1$ or $\det X^* = \exp(\pm\sqrt{2}/2)$. 

We compared Adap-RPPM with two well-established algorithms: the Difference of Convex Algorithm (DCA) and the Difference of Convex Proximal Point Algorithm (DCPPA). 
Both methods were implemented in Manopt.jl \cite{Bergmann2022}, following their standard configurations. 
The benchmarking comprised 160 test cases across multiple dimensions, with 10 random initializations per dimension generated using \texttt{rand(M)} from the \texttt{Manifolds.jl} package \cite{Bergmann2022}. 

Figure~\ref{example_03_PP} presents the performance profiles of the three methods. 
The results indicate that Adap-RPPM consistently outperforms both DCA and DCPPA in terms of computational efficiency across all tested scenarios. 

\begin{figure} 
    \centering
    \includegraphics[scale=0.4]{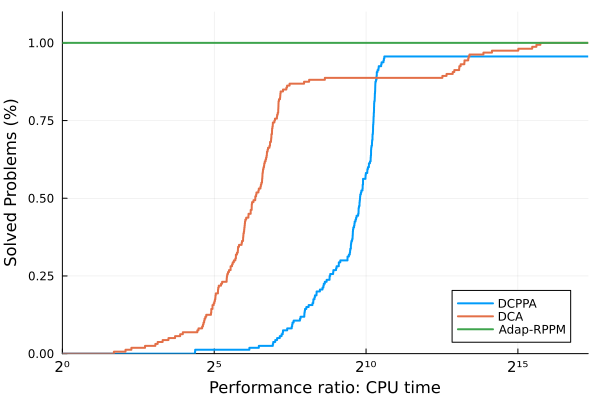}
    \caption{Performance comparison of Adap-RPPM, DCA, and DCPPA for the $f_3$ problem.}
\label{example_03_PP} 
\end{figure}

\section{Conclusions}\label{Conclusions}

In this paper, we study variants of a generalized proximal point algorithm for minimizing compositions of functions combining convex, differentiable, and lower semicontinuous terms on Riemannian manifolds. We propose two approaches: the first based on knowledge of the Lipschitz constant of the gradient of the smooth part, and the second independent of this parameter, which expands its applicability in contexts where such information is unavailable.

Our theoretical results show that both variants are convergent and admit well-established complexity bounds, in line with recent literature. Numerical experiments demonstrate the practical effectiveness of the methods in problems with complex structures, reinforcing their relevance in the study of optimization on Riemannian manifolds.

In addition to contributing to advances in proximal methods on manifolds, this work suggests promising avenues for future research. Among these are the development of accelerated versions, analysis under more general assumptions, such as Hölder-type conditions, and the investigation of applications in other contexts.

\section*{Acknowledgements}
This work was partially supported by the following grants from CNPq (National Council for Scientific and Technological Development) 401864/2022-7 .

% \appendix  %This command ends the counting of sections.
% \section*{Appendix:  Instructions for Appendices}
% If there is only one appendix, the title is Appendix. If there are appendices, the titles are Appendix A, Appendix B, and so on. Appendices should be referenced in the text. The numbering of equations,  figures, and tables should continue from the  text. For example,
% \begin{equation}
% a=b
% \end{equation}
% JOTA does not  accept electronic Supplementary Material.

%References

\bibliographystyle{plain}
\bibliography{sn-bibliography}

\end{document}